\documentclass[12pt,reqno]{amsart}
\usepackage{amssymb}
\usepackage{amsxtra}
\usepackage{mathrsfs}
\usepackage[all]{xy}
\usepackage{cite}
\usepackage{paralist}
\usepackage[hypertex]{hyperref}
\usepackage[active]{srcltx} 
\newtheorem{thm}{Theorem}[section]
\newtheorem{lm}[thm]{Lemma}

\newtheorem{pr}[thm]{Proposition}

\theoremstyle{definition}

\newtheorem{rem}[thm]{Remark}
\newtheorem{que}{Question}

\numberwithin{equation}{section}

\DeclareMathOperator{\Ker}{Ker}

\newcommand*{\wh}{\widehat}
\newcommand*{\wt}{\widetilde}

\newcommand{\ptn}{\mathbin{\widehat{\otimes}}}

\newcommand{\ad}{\mathop{\mathrm{ad}}\nolimits}

{\end{compactenum}}
\newcommand{\cO}{\mathcal{O}}

\newcommand{\CC}{\mathbb{C}}

\newcommand{\R}{\mathbb{R}}
\newcommand{\Z}{\mathbb{Z}}
\newcommand{\N}{\mathbb{N}}

\newcommand*{\fg}{\mathfrak{g}}
\newcommand*{\fl}{\mathfrak{l}}

\newcommand*{\fh}{\mathfrak{h}}

\newcommand*{\fs}{\mathfrak{s}}

\newcommand*{\M}{\mathrm{M}}

\renewcommand{\le}{\leqslant}
\renewcommand{\ge}{\geqslant}

\let \al         =\alpha
\let \be         =\beta
\let \ga         =\gamma

\let \ep         =\varepsilon      

\let \te         =\theta        
\let \io         =\iota

\let \la         =\lambda
\let \si         =\sigma

\let \De         =\Delta

\let \Si         =\Sigma

\let \phi         =\varphi

\title{Banach space representations of Drinfeld-Jimbo algebras and their complex-analytic forms}

\author{O. Yu. Aristov}
\email{aristovoyu@inbox.ru}
\keywords{Drinfeld-Jimbo algebra, Arens-Michael envelope, Banach space representation, topological Hopf algebra.}
\subjclass[2000]{Primary 17B37, 47L10, Secondary  47L55, 46H35}
\thanks{This work was supported by the RFBR grant no. 19-01-00447.}

\begin{document}

\begin{abstract}
We prove that every non-degenerate Banach space representation of the Drinfeld-Jimbo algebra $U_q(\mathfrak{g})$ of a semisimple complex Lie algebra $\mathfrak{g}$ is finite dimensional when $|q|\ne 1$. As a corollary, we find an explicit form of the Arens-Michael envelope of $U_q(\mathfrak{g})$, which is similar to that of $U(\mathfrak{g})$ obtained by Joseph Taylor in 70s. In the case when  $\mathfrak{g}=\mathfrak{s}\mathfrak{l}_2$, we also consider the representation theory of the corresponding analytic form, the Arens-Michael algebra  $\widetilde U(\mathfrak{s}\mathfrak{l}_2)_\hbar$ (with $e^\hbar=q$), and show that it is simpler than for $U_q(\mathfrak{s}\mathfrak{l}_2)$. For example, all irreducible continuous representations of $\widetilde U(\mathfrak{s}\mathfrak{l}_2)_\hbar$ are finite dimensional for every admissible value of the complex parameter $\hbar$, while $U_q(\mathfrak{s}\mathfrak{l}_2)$ has a topologically irreducible infinite-dimensional  representation when  $|q|= 1$ and $q$ is not a root of unity.
\end{abstract}

 \maketitle
 \markright{Banach space representations}

 {\small\it
\hfill To the memory of Majya Zhegalova}

 \section*{Introduction}
Besides the well-known general representation theory of semisimple complex Lie algebras, a specific theory of their Banach space representations was  also developed (see a detailed treatment of the latter in~\cite{BS01}). On the other hand, representations of Drinfeld-Jimbo algebras (quantum deformations of universal enveloping algebras) were being studied only in the algebraic context. See, e.g., the monograph \cite{KSc} for finite-dimensional representations;  a rich infinite-dimensional theory is also elaborated.

Here we are interested in Banach space representations of the Drinfeld-Jimbo algebras $U_q(\fg)$, $q\in\CC\setminus\{0,-1,1\}$,  associated with a semisimple complex Lie algebra~$\fg$ as well as that of their complex-analytic forms $\wt U(\fg)_\hbar$, $\hbar\in\CC$. (The latter series of topological algebras is defined in my article \cite{AHHFG}.)  The main results  assert that every non-degenerate Banach space representation of $U_q(\fg)$  is finite dimensional when $|q|\ne 1$ and the same is true for $\wt U(\fs\fl_2)_\hbar$   when  $e^\hbar$ is  not a root of unity  (Theorems~\ref{fdbana} and~\ref{fdbcoin}, respectively).

We also prove some results for other values of parameters. Note that for $U_q(\fg)$ three options arise naturally: $|q|\ne 1$, $q$ is a root of unity and the exceptional case  when $|q|=1$ but $q$ is not a root of unity.  The alternatives for $\wt U(\fg)_\hbar$ are more traditional: $e^\hbar$ is a root of unity or not.

We show that when $|q|=1$ and $q$ is not a root of unity, $U_q(\fg)$  admits  continuous Banach space representations that are infinite dimensional and topologically irreducible (Proposition~\ref{Tla1topirr}).  When we study the complex-analytic form, we restrict our attention to the case $\fg=\fs\fl_2$. We prove that  every continuous finite-dimensional  representation of $\wt U(\fs\fl_2)_\hbar$ is completely reducible if $e^{\hbar}$ is not a root of unity (Theorem~\ref{hbarcompred}) and  the  dimensions of  irreducible continuous
representations of $\wt U(\fs\fl_2)_\hbar$ are bounded  when  $e^\hbar$ is a root of unity (Theorem~\ref{rouirrfdhb}).
Moreover, in the first case we classify continuous irreducible representations of $\wt U(\fs\fl_2)_\hbar$  (Theorem~\ref{hbirrnotr}).

Combining results in this paper with some standard representation theory of $U_q(\fs\fl_2)$ (see \cite{Ka95} or \cite{KSc}), we obtain Tables~\ref{table:1} and~\ref{table:2} (f.d. stands for  `finite dimensional').

\begin{table}[h!]
\label{tab1}
\centering
 \caption{Banach space representations of $U_q(\fg)$ with $\fg$ semisimple}
 \label{table:1}
 \begin{tabular}{||c ||c |c||}
  \hline
 $\mathbf{q}$ & \textbf{representations} & \textbf{irr. representations}  \\ [0.5ex]
\hline\hline
$|q|\ne 1$ & f.d. compl. reducible &f.d. \\
\hline
not root, $|q|= 1$ & no restriction known &$\exists$ top. irr. inf. d. \\
\hline
root & $\exists$  inf. d. & f.d. bounded degree  \\ [0.5ex]
 \hline
 \end{tabular}
\end{table}

\begin{table}[h!]
\label{tab2}
\centering
 \caption{Banach space representations of $\wt U(\fs\fl_2)_\hbar$}
 \label{table:2}
 \begin{tabular}{||c ||c |c||}
 \hline
 $\mathbf{e^\hbar}$ & \textbf{representations} & \textbf{irr. representations}  \\ [0.5ex]
\hline\hline
not root& f.d. compl. reducible &f.d.  \\
\hline
root & $\exists$  inf .d & f.d. bounded degree  \\ [0.5ex]
 \hline
 \end{tabular}
\end{table}

As an application of our results on Banach space representations we describe the structure of the Arens-Michael envelope of $U_q(\fg)$ when $|q|\ne 1$  and the structure of $\wt U(\fs\fl_2)_\hbar$  in the case when $e^\hbar$ is  not a root of unity (Theorems~\ref{AMUq} and~\ref{AMUhbar}, respectively).
Recall that the Arens-Michael envelope of an associative algebra over~$\CC$ is a universal object connected with the problem of finding homomorphisms with range in a Banach algebra; see the definition at the end of Section~\ref{sect:UqAM}.
Considering finitely-generated associative algebras over $\CC$ as the main subject of study in Noncommutative complex affine algebraic geometry, one can treat their Arens-Michael envelopes as ``algebras of noncommutative holomorphic functions'' and thus as a possible subject of study in Noncommutative complex-analytic geometry. The same can be said for $\wt U(\fs\fl_2)_\hbar$, which is a holomorphically finitely generated algebra in the sense of Pirkovskii as defined in \cite{Pi14,Pi15}.

Finding an explicit description of the Arens-Michael envelope of a  finitely generated algebra seems easy only at first glace, with technical difficulties needing to be overcome in some cases. One of the first results was obtained by Joseph Taylor, who considered the classical (undeformed) case  and proved in \cite{T2} that the Arens-Michael envelope of $U(\fg)$ is topologically isomorphic to the direct product of a countable family of full matrix algebras, where each of the multiples corresponds to a finite-dimensional irreducible representation of~$\fg$.
For contemporary results on Arens-Michael envelopes we refer the reader to the papers of Pirkovskii \cite{Pi4,Pi11,Pir_qfree} and also the papers of  the author \cite{ArRC,ArAMN,AHHFG}. Analytic forms of quantum algebras over non-archimedean fields (including Arens-Michael envelopes) were considered in \cite{Sm18} and \cite{Du19}. Note that in the non-archimedean case such completions can be described in a more direct way than in the classical; see, e.g., \cite{Lu13}.

In his Master thesis \cite{Pe15}, Pedchenko
found the following description of the Arens-Michael envelope of $U_q(\fs\fl_2)$  in the case when  $|q|=1$.
Let $K$, $F$ and $E$ denote the standard generators of $U_q(\fs\fl_2)$ (see Section~\ref{sect:UqAM}). Then
it follows from a PBW-type theorem that
 $$
 \CC[u,z,z^{-1},v]\to  U_q(\fs\fl_2)\!:u^n z^j  v^m \mapsto F^n K^j  E^m \qquad (j\in\Z,\,n,m\in\Z_+)
 $$
determines a well-defined linear map (so-called ordered calculus) and, moreover, it can be extended to a continuous linear map
$$
 \cO(\CC\times \CC^\times\times \CC)\to  \wh U_q(\fs\fl_2)
$$
from the space of holomorphic functions to the Arens-Michael envelope. (The standard notation for the Arens-Michael envelope of an algebra $A$ is $\wh A$.)
Pedchenko proved that the latter map is a topological isomorphism when $|q|=1$. In the argument he used a method  proposed by Pirkovskii in \cite{Pir_qfree}, which is based on iterated analytic Ore extensions. This approach can be  applied only under some additional analytic conditions, which do not hold when $|q|\ne 1$, and Pedchenko left the question open in this case.

In his proof of the theorem on  the Arens-Michael envelope of $U(\fg)$  Taylor employed an analytic approach grounded on the representation theory of compact Lie groups. But in this paper we mainly use an algebraic technique.
Note that Taylor's result can be derived from the following assertion:
If $\fg$ is a semisimple complex  Lie algebra, then
the range  of any homomorphism from $U(\fg)$ to a Banach algebra is finite dimensional
\cite[\S30, Theorem~2, p.\,196]{BS01}. The proof of this assertion is essentially algebraic; it is based on the fact that
for any $\fs\fl_2$-triple, i.e., elements satisfying
\begin{equation*}
[H,E]=2E, \quad [H,F]=-2F, \quad [E,F]=H,
\end{equation*}
the relation $[[E,F],E]=2E$ holds. In a Banach algebra this relation has some algebraic consequences, which imply that every element of the completion is algebraic.
To prove the main result, Theorem~\ref{fdbana},
we use a modification of this approach and show that for any quantum $\fs\fl_2$-triple (see~\eqref{KEFrel})
and any $m\in\N$ there is a non-trivial Laurent polynomial in~$K$  that
belongs to the ideal generated by $E^m$ (Lemma~\ref{laurinid}).

The reader can find some open questions and discussion in Section~\ref{sec:remqu}.


\section{Banach space representations and the Arens-Michael envelope of $U_q(\fg)$}
\label{sect:UqAM}

Consider first the case when $\fg=\fs\fl_2$. Let $q\in\CC$, $q\ne 0$ and $q^2\ne 1$.
Recall that the quantum algebra $U_q(\fs\fl_2)$ is defined as the universal complex associative algebra generated by a quantum $\fs\fl_2$-triple $E$, $F$, $K$ (in the exponentiated form). This means that $K$ is invertible and the relations
\begin{equation}\label{KEFrel}
KEK^{-1}=q^2E,\quad KFK^{-1} =
q^{-2}F,\quad
[E,F]=\frac{K-K^{-1}}{q-q^{-1}},
\end{equation}
hold; see  \cite[\S\,3.1.1, p.\,53]{KSc}.
Consider the automorphism  $\si$ of $\CC[K,K^{-1}]$ determined by $\si(K)\!:=q^2K$.
It is easy to see that
\begin{equation}\label{FRsi}
FR=\si(R)F
\end{equation}
for every $R\in \CC[K,K^{-1}]$. This equality will be useful in what follows.

Now let $\fg$ be an arbitrary  semisimple complex  Lie algebra  and $q\in\CC\setminus\{0\}$. Denote the  rank of $\fg$ by~$l$. The condition $q^{2d_j}\ne 1$ is also imposed for some non-zero $d_1,\ldots,d_l$ given by the weight theory.
Recall that the Drinfeld-Jimbo algebra $U_q(\fg)$ is the algebra with  generators
$E_j$, $F_j$, $K_j$ and $K_j^{-1}$ ($j=1,\ldots l$) subject to a number of relations. For the complete list see \cite[\S\,6.1.2, p.\,161, (12)--(16)]{KSc}). For our purposes we need mainly the facts that~$K_j$ pairwise commute and  for each~$j$ the elements $E_j$, $F_j$, $K_j$ form a quantum $\fs\fl_2$-triple, namely,  the relations in~\eqref{KEFrel} are satisfied with $E$, $F$, $K$ and~$q$ replaced by $E_j$, $F_j$, $K_j$ and $q^{d_j}$, respectively. The other relations  are the commutation relations $[E_i,F_j]=0$, $i\ne j$, and the so-called quantum Serre relations.

\subsection*{The case when $|q|\ne 1$}

We begin with a theorem similar to the assertion that the range of
any homomorphism from $U(\fg)$ to a Banach algebra is finite dimensional
(see \cite[\S30, Theorem~2, p.\,196]{BS01}). At the end of the section,  we use this result to describe the structure of the Arens-Michael envelope of $U_q(\fg)$ in the case when $|q|\ne 1$.

When we speak about a representation on a Banach space, we always mean a representation by bounded operators (`topological representation' in the terminology of~\cite{X2}.)

\begin{thm}\label{fdbana}
Let $\fg$ be a   semisimple complex Lie algebra and $|q|\ne 1$.

\emph{(A)}~The range of any homomorphism from $U_q(\fg)$ to a Banach algebra is finite dimensional.

\emph{(B)}~Every non-degenerate representation of $U_q(\fg)$  on a Banach space is finite dimensional.
\end{thm}
The argument splits into two parts, analytic and algebraic. The analytic part (Lemma~\ref{invqnil}) is simple but the algebraic part (Proposition~\ref{fdbanpr}) is a bit  more involved.

\begin{lm}\label{invqnil}
Let $a$ and $c$ be elements of a Banach  algebra.
Suppose that~$a$ is invertible
and $aca^{-1}=\ga c$ for some $\ga\in\CC$ such that $|\ga|\ne 1$ and $\ga\ne 0$. Then~$c$ is nilpotent.
\end{lm}
\begin{proof}
Assume the opposite, i.e., that $c^n\ne 0$ for  every $n\in\N$. Let $\|\cdot\|$ denote the norm on the Banach algebra.
Since  $ac^na^{-1}=\ga^n c^n$ and $\|\cdot\|$ is submultiplicative, we have that $\|\ga^n c^n\|\le \|a\|\,\|c^n\|\,\|a^{-1}\|$ and so $|\ga|^n \le \|a\|\,\|a^{-1}\|$ for all $n\in\N$. Therefore $|\ga|\le 1$. Letting $d\!:=aca^{-1}$ and using the equality $a^{-1}da=\ga^{-1} d$, we obtain similarly that $|\ga|^{-1}\le 1$. This contradicts the hypothesis.
\end{proof}

Let $\be_1,\ldots, \be_n$ be the positive roots of~$\fg$ and let $E_{\be_r}$ and $F_{\be_r}$ ($r=1,\ldots,n$) be the corresponding root elements of $U_q(\fg)$ \cite[\S\,6.2.3, p.\,175, (65)]{KSc}.

We use the following proposition twice, right now in the proof of Theorem~\ref{fdbana} and in~\S\,\ref{sechbar}.
\begin{pr}\label{fdbanpr}
Suppose that $q$ is not a root of unity and $\pi$ is a homomorphism from $U_q(\fg)$ to some associative algebra.
If $\pi(E_{\be_r})$ and $\pi(F_{\be_r})$ are nilpotent for every~$r$, then the range of $\pi$ is finite dimensional.
\end{pr}

Assuming for the moment that the proposition holds, we can easily prove the theorem.
\begin{proof}[Proof of Theorem~\ref{fdbana}]
(A)~Fix $r\in\{1,\ldots,n\}$. For $\la=n_1\al_1+\cdots+n_l\al_l$, where $n_j\in\Z$ and $\al_1,\ldots,\al_l$ are the simple roots corresponding to $K_1,\ldots,K_l$, put  $K_\la = K_1^{n_1}\cdots K_l^{n_l}$. Then by
\cite[\S\,6.2.3, p.\,176, Proposition 23(iii)]{KSc},
$$
K_\la E_{\be_r} K_\la^{-1} = q^{( \la,\be_r)} E_{\be_r}\quad\text{and}\quad K_\la F_{\be_r} K_\la^{-1}  =
q^{-( \la,\be_r)} F_{\be_r},
$$
where $(\cdot,\cdot)$ is the restriction of the Killing form to $\fh_\R$ (see definitions in \cite[p.\,157--158]{KSc}).
Putting $\la=\be_r$ we have $|q^{( \la,\la)}|\ne 1$ because $|q|\ne1$ and $(\cdot,\cdot )$ is a positive definite bilinear form on $\fh_\R$.

Let~$\pi$ be a homomorphism from $U_q(\fg)$ to a Banach algebra. It follows
from Lemma~\ref{invqnil} with $\ga=q^{( \la,\la)}$ that $\pi(E_{\be_r})$ and $\pi(F_{\be_r})$ are nilpotent for every~$r$. Since~$q$ is not a root of unity, we can apply Proposition~\ref{fdbanpr}.

Part~(B) follows immediately from Part~(A).
\end{proof}

In the proof of Proposition~\ref{fdbanpr} we need  auxiliary lemmas. Suppose that  $E$, $F$ and $K$ are elements of some algebra that satisfy the relations in~\eqref{KEFrel} for given~$q$.

\begin{lm}\label{laurinid}
Let $m\in\N$. Suppose that for each $n\in\{1,\ldots,m\}$ there is a non-trivial Laurent polynomial $P_n$ in~$K$  such that
\begin{equation}\label{commEnF}
[E^n,F]=E^{n-1}P_n.
\end{equation}
Then there is a non-trivial Laurent polynomial in~$K$  that
belongs to the ideal generated by~$E^m$.
\end{lm}
\begin{proof}
Denote the ideal generated by $E^m$ by $J$. We claim that for each $n\in\{0,\ldots,m-1\}$ there is a non-trivial  $R_n\in\CC[K,K^{-1}]$ such that $E^nR_n\in J$. This claim with $n=0$ is exactly the assertion of the lemma.

We proceed by reverse induction. Put $R_{m-1}=P_m$.
By~\eqref{commEnF}, the claim holds when $n=m-1$.

Assume now that the claim has been proved for some  $n\le m-1$ and put $$R_{n-1}\!:=P_n R_{n}\si^{-1}(R_{n}).$$
Applying \eqref{commEnF} and then \eqref{FRsi} with $R=R_{n}\si^{-1}(R_{n})$, we have
\begin{multline*}
E^{n-1}R_{n-1}=
E^{n-1}P_n R_{n}\si^{-1}(R_{n})=[E^n,F]R_{n}\si^{-1}(R_{n})\\
=E^nFR_{n}\si^{-1}(R_{n}) -FE^nR_{n}\si^{-1}(R_{n})=E^n\si(R_{n})R_nF -FE^nR_{n}\si^{-1}(R_{n}).
\end{multline*}
Note that $\si(R_{n})$ and $R_n$ commute.
Since  by the induction hypothesis, $E^nR_{n}$ belongs to~$J$, so is $E^{n-1}R_{n-1}$ and the claim is proved.
\end{proof}

\begin{lm}\label{algel}
Let $A$ be an associative algebra  generated by an element~$a$. Then $a$ is algebraic (i.e., there is a non-trivial polynomial~$p$ such that $p(a)=0$) if and only if~$A$ is finite dimensional.
\end{lm}
\begin{proof}
It suffices to note that both conditions, that $a$ is algebraic and that $A$ is finite dimensional, are equivalent to the fact that there is $k\in\N$ such that $a^k$ is a linear combination of $1,a,\ldots,a^{k-1}$.
\end{proof}

\begin{proof}[Proof of Proposition~\ref{fdbanpr}]
Fix~$m$ such that $\pi(E_{\be_r})^m=\pi(F_{\be_r})^m=0$ for every~$r$ and denote by~$I$ the ideal of $U_q(\fg)$ generated by $\{E_{\be_r}^m,\,F_{\be_r}^m;\,r=1,\ldots,n\}$. Since~$\pi$ factors on the quotient homomorphism~$U_q(\fg)\to U_q(\fg)/I$, it suffices to show that the image of $U_q(\fg)/I$ under the induced homomorphism is finite dimensional. The set
$$
\{F_{\be_1}^{r_1}\cdots F_{\be_n}^{r_n} K_1^{k_1} \cdots K_l^{k_l}E_{\be_n}^{s_n} \cdots E_{\be_1}^{s_1}\},
$$
where $r_1,\ldots,r_n$, $s_1,\ldots,s_n$ run through~$\Z_+$ and $k_1,\ldots,k_l$ run through~$\Z$,
is a linear basis of $U_q(\fg)$ (see \cite[\S\,6.2.3, p.\,176, Theorem~24]{KSc} for the statement and \cite{Lu93} for the proof). So it suffices to show  that
the subalgebra~$A$ of $U_q(\fg)/I$ generated by
$$
\{K_j+I,\,K_j^{-1}+I:\,j=1,\ldots,l\}
$$
is finite dimensional. Moreover, since $K_j+I$ pairwise commute, it suffices to show
that the subalgebra generated by~$K_j+I$ and~$K_j^{-1}+I$ is finite dimensional for every $j$.

Indeed, since $E_j$, $F_j$, $K_j$ form a quantum $\fs\fl_2$-triple with parameter $q_j\!:=q^{d_j}$, we have for any $m\ge1$ that
\begin{equation}\label{commEmFts}
[E_j^m,F_j]=E_j^{m-1}(t_{mj} K_j-s_{mj} K_j^{-1}),
\end{equation}
where
$$
t_{mj}\!:=\frac{q_j(q_j^{2m}-1)}{(q_j^2-1)^2},\qquad
s_{mj}\!:=\frac{q_j^{-1}(q_j^{-2m}-1)}{(q_j^{-2}-1)^2};
$$
see, e.g., \cite[Lemma~VI.1.3]{Ka95}.
Since $q$ is not a root of unity, all the coefficients $t_{mj}$ and $s_{mj}$ are non-trivial. Therefore by
Lemma~\ref{laurinid},   for every~$j$ there is a non-trivial Laurent polynomial in~$K_j$  that
belongs to~the ideal generated by $E_j^m$. Since $\{E_1,\ldots,E_l\}\subset \{E_{\be_1} \cdots E_{\be_n}\}$ (see \cite[\S\,I.6.8, p.\,52]{BG02}), this Laurent polynomial is also in~$I$.
Hence $K_j+I$ is algebraic in $U_q(\fg)/I$  and so by
Lemma~\ref{algel}, the subalgebra of $U_q(\fg)/I$ generated by~$K_j+I$ and~$K_j^{-1}+I$ is finite dimensional.
\end{proof}

Now we can prove an analogue of Taylor's theorem on the Arens-Michael envelope of $U(\fg)$.
Recall that an \emph{Arens-Michael algebra} is a complete topological algebra whose topology
can be determined by a system of submultiplicative prenorms $\|\cdot\|$, i.e., the inequality $\|ab\|\le\|a\|\,\|b\|$ holds for every~$a$ and~$b$.
An \emph{Arens-Michael envelope} of an associative algebra~$A$ over~$\CC$
is  a pair $(\wh A, \io_A)$, where
$\wh A$ is an Arens-Michael algebra  and $\io_A$ is a
homomorphism $A \to \wh A$, such that for any
Arens-Michael algebra $B$ and for each homomorphism
$\phi\!: A \to B$ there exists a unique continuous homomorphism
$\widehat\phi\!:\wh A \to B$ making the diagram
\begin{equation*}
  \xymatrix{
A \ar[r]^{\io_A}\ar[rd]_{\phi}&\wh A\ar@{-->}[d]^{\widehat\phi}\\
 &B\\
 }
\end{equation*}
commutative \cite[Chapter~5]{X2}. In fact, it suffices to check this property only for Banach algebras.
Note that $\wh A$ is topologically isomorphic to the completion of~$A$ with respect to the topology determined by
all submultiplicative prenorms.

Let $\fg$ be a   semisimple complex Lie algebra and
$\Si_q$ be the set of the equivalence
classes of irreducible finite-dimensional representations of $U_q(\fg)$ for given $q$.
Then for $\si\in\Si_q$ we have a homomorphism
$U_q(\fg)\to \M_{d_\si}(\CC)$, where $d_\si$ is the dimension of $\si$ and $\M_{d_\si}(\CC)$ is the algebra of matrices of order~$d_\si$.
Denote by $\io$ the corresponding homomorphism
$$
U_q(\fg)\to \prod_{\si\in\Si_q}\M_{d_\si}(\CC).
$$

\begin{thm}\label{AMUq}
Let $\fg$ be a semisimple complex  Lie algebra.
If $|q|\ne 1$, then the algebra
$\prod_{\si\in\Si_q}\M_{d_\si}(\CC)$ endowed with the direct product topology together with $\io$ is the Arens-Michael envelope of $U_q(\fg)$.
\end{thm}
\begin{proof}
It suffices to show that every homomorphism $\phi$ from  $U_q(\fg)$ to a Banach algebra factors through $\prod_{\si\in\Si_q}\M_{d_\si}(\CC)$. Denote the range of~$\phi$ by~$B$.
By Theorem~\ref{fdbana}, $B$ is finite dimensional and so it is a finite-dimensional Banach algebra.
Then $B$ becomes a $U_q(\fg)$-module with respect to the action given by $a\cdot b\!:=\phi(a)b$. Since $q$ is not a root of unity, any finite-dimensional $U_q(\fg)$-module is completely reducible \cite[Theorem~2]{Ro88}. Therefore $\phi$ factors through some finite product of algebras of the form $\M_{d_\si}(\CC)$ and hence through $\prod_{\si\in\Si_q}\M_{d_\si}(\CC)$.
\end{proof}

In the rest of the article we suppose that $\fg=\fs\fl_2$.

\begin{rem}\label{Uotc2}
Comparing the lists of irreducible finite-dimensional representations of $U(\fs\fl_2)$ and $U_q(\fs\fl_2)$ (see~\eqref{Tnep} below),
we have that for any such representation of $U(\fs\fl_2)$  there are exactly two such representations of $U_q(\fs\fl_2)$. So, in view of Theorem~\ref{AMUq}, we can identify
$\wh U_q(\fs\fl_2)$ with  $\wh U(\fs\fl_2)\otimes \CC^2$ when $|q|\ne 1$. Of course, this isomorphism is not canonical.
\end{rem}

\subsection*{The case when $|q|=1$ but $q$ is not a root of unity}

Suppose that $\fg=\fs\fl_2$.
A power series description of  the Arens-Michael envelope of $U_q(\fs\fl_2)$ in the case when $|q|=1$ is given by Pedchenko in \cite{Pe15}. But it says nothing about Banach space representations.

We recall some representation theory of $U_q(\fs\fl_2)$. If $q$ is not a root of unity, then as mentioned above,  any finite-dimensional representation of $U_q(\fs\fl_2)$ is completely reducible. Moreover, any irreducible finite-dimensional representation of $U_q(\fs\fl_2)$
is associated with a homomorphism of the form
\begin{equation}\label{Tnep}
U_q(\fs\fl_2)\to \M_{n+1}(\CC)\!:\, E\mapsto E_{n,\ep},\quad  F\mapsto F_{n,\ep},\quad K\mapsto K_{n,\ep},
\end{equation}
 where $\ep=\pm1$, $n\in\Z_+$,
 $$
 E_{n,\ep}=\ep
\begin{pmatrix}
 0 &[n]_q &0&\ldots&0\\
 0 &0 &[n-1]_q&\ldots&0\\
 \vdots &\ddots &\ddots&\ddots&\vdots\\
 0 &0 &\ddots&\ddots&1\\
 0 &0 &\ldots&0&0\\
   \end{pmatrix},\quad
   F_{n,\ep}=
\begin{pmatrix}
 0 &0 &\ldots&0&0\\
 1 &0 &\ldots&0&0\\
 0 &[2]_q &\ddots&0&0\\
 \vdots &\ddots &\ddots&\ddots&\vdots\\
 0 &0 &\ldots&[n]_q&0\\
   \end{pmatrix},
$$
$$
 K_{n,\ep}=\ep
\begin{pmatrix}
 q^n &0 &\ldots&0&0\\
 0 &q^{n-2} &\ldots&0&0\\
 \vdots &\ddots &\ddots&\ddots&\vdots\\
 0 &0&\ldots&q^{-n+2}&0\\
 0 &0 &\ldots&0&q^{-n}\\
   \end{pmatrix};
$$
here
$$
[n]_{q}\!:=\frac{q^n-q^{-n}}{q-q^{-1}};
$$
see \cite[\S\,3.2, p.\,62, Proposition~9]{KSc}. So we have an infinite series of irreducible (finite-dimensional) Banach space representations. We now suppose that $|q|=1$ but $q$ is not a root of unity and consider  completions of Verma modules.

Let $\la\in\CC\setminus\{0\}$ and $V(\la)$ denote the Verma module  of $U_q(\fs\fl_2)$ as defined, e.g., in  \cite[p.\,129, Lemma VI.3.6]{Ka95}. Namely, $V(\la)$ is a linear space  with basis $\{e_n\!: \,n\in\N\}$ and the generators of $U_q(\fs\fl_2)$ are represented by infinite matrices:
$$
E\mapsto
\begin{pmatrix}
 0 &-[1]_{q,\la} &0&\ldots&0&\ldots\\
 0 &0 &-[2]_{q,\la}&\ldots&0&\ldots\\
 \vdots &\ddots &\ddots&\ddots&\vdots&\vdots\\
 0 &0 &\ddots&\ddots&-[n]_{q,\la}&\ddots\\
  \vdots &\vdots&\ldots&\vdots&\vdots&\ddots
   \end{pmatrix}\,,
$$
$$
F\mapsto
\begin{pmatrix}
 0 &0 &\ldots&0&\ldots\\
 1 &0 &\ldots&0&\ldots\\
 0 &[2]_q &\ddots&0&\ldots\\
 \vdots &\ddots &\ddots&\ddots&\vdots\\
 0 &0 &\ldots&[n]_q&\ddots\\
 \vdots &\vdots&\ldots&\vdots&\ddots
   \end{pmatrix}\,,\quad
K\mapsto
\begin{pmatrix}
 \la &0 &\ldots&0&\ldots\\
 0 &\la q^{-2} &\ldots&0&\ldots\\
 \vdots &\vdots &\ddots&\vdots&\vdots\\
 0 &0&\ldots&\la q^{-2n}&\vdots\\
 \vdots &\vdots&\ldots&\vdots&\ddots
   \end{pmatrix}\,,
$$
where
\begin{equation}\label{nqla}
[n]_{q,\la}\!:=\frac{q^n\la^{-1}-q^{-n}\la}{q-q^{-1}}.
\end{equation}
(Note that $[n]_q=[n]_{q,1}$.)

Let $p\in[1,+\infty)$. It is easy to see that this representation can be extended to a representation on the Banach space $\ell^p$ if and only if $|q|=1$. We denote this extension by~$S_{\la,p}$. So we have a series of infinite-dimensional Banach space representations of $U_q(\fs\fl_2)$ in contrast to the case when $|q|\ne 1$.

Recall that a representation of an algebra on a Banach space is said to be \emph{topologically irreducible} if there is no proper closed invariant subspace.
In the following proposition we assume that $p=2$ because the argument uses a result of Wermer on  operators on a Hilbert space; see \cite[Theorem~4]{We52}.

\begin{pr}\label{Tla1topirr}
Suppose that $q\in\CC$ is not a root of unity, $|q|=1$ and $\la\ne0$. Then $S_{\la,2}$ is topologically irreducible if and only if~$q^2$ is not a root of~$\la^2$.
\end{pr}
\begin{proof}
Note that $S_{\la,2}(K)$ is a normal operator and the Hilbert space~$\ell^2$ has an orthonormal basis of eigenvectors of~$S_{\la,2}(K)$. Since the modulus of every eigenvalue equals $|\la|$, spectral  synthesis  holds  for $S_{\la,2}(K)$ \cite[Theorem~4]{We52}, i.e., every non-trivial closed invariant subspace coincides with the closure of all eigenvectors contained in~it.

Since $q$ is not a root of unity, all the numbers $[n]_q$ are non-zero.
It follows from the form of $S_{\la,2}(F)$ that if a non-trivial closed invariant subspace contains~$e_n$ for some $n\in\N$, then it also contains~$e_{n+1}$.
On the other hand, it follows from the form of $S_{\la,2}(E)$ that the closure of the linear span of $\{e_k\!:\,k\ge n+1\}$ is invariant if and only if
$[n]_{q,\la}=0$.
Thus, a non-zero proper closed invariant subspace exists if and only if there is $n$ such that $q^{2n}=\la^2$.
\end{proof}

\subsection*{The case when  $q$ is a root of unity}

Suppose now that $q$ is a root of unity.
It is well known that then $U_q(\fs\fl_2)$ has big centre and this implies that all  irreducible representations of $U_q(\fs\fl_2)$  are finite dimensional \cite[\S\,3.3, Corollary~15, p.\,67]{KSc}. The following proposition holds in contrast to the case when $|q|\ne 1$.

\begin{pr}
Let $q$ be a root of unity.
Then there are infinite-dimensional Banach space representations of $U_q(\fs\fl_2)$.
\end{pr}
\begin{proof}
It suffices to show that there is a homomorphism  to the Banach algebra with infinite-dimensional range.

Let $T_{ab\la}$ be the three-parameter family of $p'$-dimensional representations ($a,b,\la\in \CC$ and $\la\ne 0$) of $U_q(\fs\fl_2)$, where $p'\in\N$, as defined in \cite[\S\,3.3, p.\,68, (34)]{KSc}.
We consider each representation space as a finite-dimensional Hilbert space and assume that the standard basis is normed to one.

Fix $a$ and $b$ and take an infinite compact subset~$K$ of $\CC\setminus\{0\}$.
Then there is $C>0$ such that the norms of $T_{ab\la}(E)$, $T_{ab\la}(F)$ and $T_{ab\la}(K)$ are at most $C$ when $\la$ runs $K$. So we have a homomorphism from $U_q(\fs\fl_2)$ to the Banach algebra
$C(K,\M_{p'}(\CC))$ of matrix-valued continuous functions. It is evident from the explicit form of $T_{ab\la}$  that the range is infinite dimensional.
\end{proof}

\section{Banach space representations and the structure of $\wt U(\fs\fl_2)_\hbar$}
\label{sechbar}

In this section we study representations of the Arens-Michael algebra $\wt U(\fs\fl_2)_\hbar$ introduced
in \cite[\S\,5]{AHHFG}.

Let $\hbar\in\CC$ and $\sinh\hbar\ne 0$. Then $\wt U(\fs\fl_2)_\hbar$ denotes
the universal Arens-Michael algebra generated by $E$, $F$, $H$  subject to relations
\begin{equation}\label{EFHqun}
[H,E]=2E, \quad [H,F]=-2F, \quad [E,F]=\frac{\sinh \hbar H}{\sinh\hbar}\,.
\end{equation}

(The term `universal' means that for any Arens-Michael algebra~$B$ containing elements that satisfy~\eqref{EFHqun} there is a unique continuous homomorphism $\wt U(\fs\fl_2)_\hbar\to B$ sending the generators to that elements.
The universal Arens-Michael algebra is isomorphic to the quotient of the algebra of free entire functions over the closed two-sided ideal generated by the corresponding identities. See, e.g., \cite{Pi15} for the algebras of free entire functions and their quotients.)

The algebra $\wt U(\fs\fl_2)_\hbar$ can be endowed with a structure of a topological Hopf algebra but we do not need it here; see details in~\cite{AHHFG}. Note also that one can define similarly the algebra $\wt U(\fg)_\hbar$ for every  semisimple complex  Lie algebra~$\fg$ [ibid., Remark~5.5]. But here we consider only the case when $\fg=\fs\fl_2$.

First, we discuss a connection between $\wt U(\fs\fl_2)_\hbar$ and $U_q(\fs\fl_2)$.

\begin{lm}\label{EFHKrel}
Let $E$, $F$ and $H$ be elements of an Arens-Michael algebra satisfying~\eqref{EFHqun}. Put
\begin{equation}\label{Kqdef}
K\!:=e^{\hbar H}\quad\text{and}\quad q\!:=e^{\hbar}.
\end{equation}
Then $E$, $F$ and $K$ satisfy~\eqref{KEFrel}.
\end{lm}
\begin{proof}
We recall that the well-known formula
\begin{equation}\label{adhQ}
\ad h(Q)(T) = \sum_{n=1}^\infty \frac{1}{n!}  (\ad Q)^n(T)h^{(n)}(Q)
\end{equation}
holds for elements $Q$ and $T$ of an Arens-Michael algebra and an entire function~$h$; cf. \cite[\S\,15,
p.\,82, Corollary~1]{BS01}. (Here $\ad Q(T)\!:=[Q,T]$.)

It follows from $[H,E]=2 E$ that
$(\ad H)^n(E)=2^n E$. Hence by~\eqref{adhQ},
$$
(\ad h(H))E = \sum_{n=1}^\infty \frac{1}{n!}\,  2^n\,E\, h^{(n)}(H)=E\,(h(H+2)-h(H)) ,
$$
i.e., $h(H)E=Eh(H+2)$.
Therefore,
$$
KE=e^{\hbar H}E=E\,e^{\hbar (H+2)}=q^2EK.
$$
Similarly, we have  $KF=q^{-2}FK$. The last relation in~\eqref{KEFrel} is trivial.
\end{proof}

Thus it follows from Lemma~\ref{EFHKrel} that we have a well-defined homomorphism determined by
\begin{equation}\label{homthe}
\te\!:U_q(\fs\fl_2)\to \wt U(\fs\fl_2)_\hbar\!:E\to E,\quad F\to F,\quad K\to e^{\hbar H}.
\end{equation}
(For simplicity of notation, we denote the generators~$E$ and~$F$ in $U_q(\fs\fl_2)$ and $\wt U(\fs\fl_2)_\hbar$ by the same letters.)

The following theorem is the first of two main results in this section. The second is Theorem~\ref{AMUhbar}.

\begin{thm}\label{mainthhbar}
Let  $\sinh\hbar\ne 0$.
Every irreducible continuous Banach space
representation of $\wt U(\fs\fl_2)_\hbar$ is finite dimensional.
\end{thm}
We consider separately the cases when $e^{\hbar}$ is not a root of unity and when it is.

\begin{thm}\label{fdbcoin}
Suppose that $e^{\hbar}$ is not a root of unity.

\emph{(A)}~The range of any continuous homomorphism from $\wt U(\fs\fl_2)_\hbar$ to a Banach algebra is finite dimensional.

\emph{(B)}~Any non-degenerate continuous representation of $\wt U(\fs\fl_2)_\hbar$  on a Banach space is finite dimensional.
\end{thm}

We need two lemmas. The first improves Lemma~\ref{algel}.
\begin{lm}\label{algel2}
Let $B$ be a Banach algebra (algebraically) generated by an element~$b$.
If there is a non-trivial function~$h$ holomorphic in a neighbourhood~$U$ of the spectrum of~$b$ and satisfying $h(b)=0$, then~$B$ is finite dimensional.
\end{lm}
\begin{proof}
Since the spectrum $b$ is compact, we can assume that~$U$ contains only a finite number of zeros of~$h$. Therefore $h=ph_1$, where~$p$ is a polynomial and~$h_1$ is a function holomorphic and non-vanishing in~$U$. Then $h_1(b)$ is invertible and so $p(b)=0$. An application of Lemma~\ref{algel} completes the proof.
\end{proof}

\begin{lm}
\label{fdpolH}
Suppose that a unital associative algebra is generated by elements $e$,~$f$ and~$h$  such that
$$
[h,e]=\al e, \quad [h,f]=\be f, \quad [e,f]=p(h).
$$
where $\al,\be\in\CC$ and  $p$ is a polynomial. Then  the linear span $L$ of
$\{e^{j} h^{k} f^{n}\}$, where $j,k,n$ run through $\Z_+$, coincides with the whole algebra.
\end{lm}
\begin{proof}
It suffices to show that every product of the form
$e^{j_1} h^{k_1} f^{n_1}e^{j_2} h^{k_2} f^{n_2}$, where the exponents are in $\Z_+$, is contained in~$L$. It is obvious that $fe\in L$ and it can be shown by induction (first in~$j_2$ and next in~$n_1$) that the same is true for $f^{n_1}e^{j_2}$.
It is easy to see that the subalgebra generated by $e$ and $h$ (or $f$ and $h$)  coincides with the linear span of $\{e^{j} h^{k} \}$, $j,k\in\Z_+$, (the linear span of $\{ h^{k} f^{n}\} $, $k,n\in\Z_+$, respectively). Therefore the product under consideration is in~$L$.
\end{proof}

\begin{proof}[Proof of Theorem~\ref{fdbcoin}]
(A)~Let $\pi$ be a continuous homomorphism from $\wt U(\fs\fl_2)_\hbar$ to a Banach algebra~$B$.
It is not hard to see that the relations $[H,E]=2E$ and $[H,F]=-2F$ imply that $\pi(E)$ and $\pi(F)$ are nilpotent (cf. \cite[Example 5.1]{Pir_qfree}).

Define $q$ and $K$  as in~\eqref{Kqdef} and consider the homomorphism $\te$ as in~\eqref{homthe}. Since $q$ is not a root of unity,
Proposition~\ref{fdbanpr} implies that the range of $\pi\te$ is finite dimensional and so is the subalgebra generated by~$\pi(K)$. By Lemma~\ref{algel},
$\pi(K)$ is an algebraic element, i.e., there a polynomial $p$ such that $p(\pi(K))=0$.

Since $\pi$ is continuous, $e^{\hbar\pi(H)}=\pi(K)$ and so $p(e^{\hbar\pi(H)})=0$. It follows from Lemma~\ref{algel2} that the algebra~$B_0$ generated by $\pi(H)$ is finite dimensional. Hence~$B_0$ is closed and then is a Banach subalgebra of~$B$. Therefore $\sinh \hbar \pi(H)\in B_0$. So there a polynomial $p_0$ such that $\sinh \hbar \pi(H)=p_0(\pi(H))$. Thus $\pi(E)$, $\pi(F)$ and $\pi(H)$ satisfy the hypotheses of Lemma~\ref{fdpolH} and so the linear span of
$\{\pi(E)^{j} \pi(H)^{k} \pi(F)^{n}\}$, where $j,k,n\in\Z_+$, coincides with the subalgebra $B_1$ generated by these three elements. Since $\pi(E)$ and $\pi(F)$ are nilpotent and $\pi(H)$ is algebraic, $B_1$ is finite dimensional and hence closed. Finally, note that $E$, $F$ and $H$ are topological generators of $\wt U(\fs\fl_2)_\hbar$ and so the range of~$\pi$ equals~$B_1$.

Part~(B) follows immediately from Part~(A).
\end{proof}

Now we turn to the case when $e^{\hbar}$ is a root of unity.
It has been mentioned above that under the same assumption on $q$ every irreducible representation of $U_q(\fs\fl_2)$  is finite dimensional \cite[\S\,3.3, Corollary~15, p.\,67]{KSc}.  Slightly changing the argument, we obtain the same assertion for $\wt U(\fs\fl_2)_\hbar$.

\begin{thm}\label{rouirrfdhb}
Suppose that  $e^\hbar$ is a root of unity. Then  every irreducible continuous
representation of $\wt U(\fs\fl_2)_\hbar$ is finite dimensional. Moreover, there is $s\in\N$ such that all irreducible continuous representations have dimension at most~$s$.
\end{thm}
\begin{proof}
Let $K\!:= e^{\hbar H}$, $q\!:=e^\hbar$ and $\te\!:U_q(\fs\fl_2)\to \wt U(\fs\fl_2)_\hbar$ be the  homomorphism given by~\eqref{homthe}.
Suppose that $q$ is a root of unity of degree~$d$.
Put $s\!:=d$  when $d$ is odd and $s\!:=d/2$ when $d$ is even.
It is well known that $E^s$, $F^s$ and $K^s$ are in the
center of $U_q(\fs\fl_2)$ \cite[p.\,134, Lemma~IV.5.3]{Ka95}.
Since~$\te(K^s)$ obviously commutes with~$H$, it is in the center of $\wt U(\fs\fl_2)_\hbar$.

Let $T$ be an irreducible continuous representation of $\wt U(\fs\fl_2)_\hbar$.
It follows from the continuity and Shur's lemma that
$e^{s\hbar T(H)}=T\te(K^s)=\la$ for some $\la\in\CC$. By Lemma~\ref{algel2}, the subalgebra generated by $T(H)$ is finite dimensional and hence closed.  Therefore there is a polynomial $p$ such that $T(\sinh\hbar H/\sinh\hbar)=p(T(H))$.
Since  $T(E)$ and $T(F)$ are nilpotent (cf. the proof of Theorem~\ref{fdbcoin}), it follows from Lemma~\ref{fdpolH} that the range of $T$ is finite dimensional.

To show that there is no irreducible  finite-dimensional representation  of
dimension greater than~$s$ we use the same argument as in
\cite[p.\,134. Proposition VI.5.2]{Ka95}. Indeed, assume to the contradiction that $T$ is such a representation.
Two cases can occur: there  exists  a  non-zero  eigenvector  of $T(H)$ such that
$T(F)v = 0$ or not. In the first case denote by $V'$ the linear span of $\{v,T(E)v,\ldots,  T(E^{s-1})v\}$ and in the second case take a  non-zero  eigenvector  of $T(H)$ such that $T(F)v \ne 0$ and denote by $V''$ the linear span of  $\{v,T(F)v,\ldots,  T(F^{s-1})v\}$. It follows from $[H,E]=2E$ and $[H,F]=-2F$ that $V'$ and $V''$ respectively, are invariant under $T(H)$. Moreover, as it is proved in [ibid.], in the first and second cases, $V'$ and $V''$, respectively, are invariant under both $T(E)$ and $T(F)$. Thus we have an invariant subspace of dimension~$s$.
\end{proof}

Combining Theorem~\ref{rouirrfdhb} with Part~(B) of Theorem~\ref{fdbcoin}, we immediately deduce  Theorem~\ref{mainthhbar}.

\subsection*{A classification of irreducible continuous representations of $\wt U(\fs\fl_2)_\hbar$ in the case when $e^{\hbar}$ is not a root of unity}

Let $n\in\Z_+$, $k\in\Z$ and $\ep\in\{-1,1\}$. Put
$$
H_{n,k,\ep}\!:=
\begin{pmatrix}
 n+r_{k,\ep} &0 &\ldots&0&0\\
 0 &n-2+r_{k,\ep} &\ldots&0&0\\
 \vdots &\ddots &\ddots&\ddots&\vdots\\
 0 &0&\ldots&-n+2+r_{k,\ep}&0\\
 0 &0 &\ldots&0&-n+r_{k,\ep}\\
   \end{pmatrix}\,,
$$
where
$r_{k,1}\!:=2k\pi \hbar^{-1}i$ and $r_{k,-1}\!:=(2k+1)\pi \hbar^{-1}i$.

It is not nard to see that $E_{n,\ep}$, $F_{n,\ep}$ (see~\eqref{Tnep}) and $H_{n,k,\ep}$ satisfy the relations in~\eqref{EFHqun}. Recall that $\wt U(\fs\fl_2)_\hbar$ is defined as a universal algebra; so we have a continuous $(n+1)$-dimensional representation of it determined by
\begin{equation}\label{defTnkep}
 E\mapsto E_{n,\ep},\quad  F\mapsto F_{n,\ep},\quad H\mapsto H_{n,k,\ep}.
\end{equation}
We denote this representation by $T_{n,k,\ep}$ and the representation of $U_q(\fs\fl_2)$ defined in~\eqref{Tnep} by $T_{n,\ep}$. It is easy to see that $T_{n,\ep}= T_{n,k,\ep} \te$.

\begin{rem}
When $|e^{\hbar}|=1$ and $e^{\hbar}$ is not a root of unity, the Banach space representation theory of $\wt U(\fs\fl_2)_\hbar$ is quite different from that of $U_q(\fs\fl_2)$. For example, the infinite-dimensional representation $S_{\la,2}$ in Proposition~\ref{Tla1topirr} cannot be modified  in the same way as $T_{n,\ep}$ was obtained  from  $T_{n,k,\ep}$ because in this case we get a matrix with unbounded sequence of eigenvalues.
Moreover, $\wt U(\fs\fl_2)_\hbar$  has no continuous infinite-dimensional representation on a  Banach space as is shown in Theorem~\ref{fdbcoin}.
\end{rem}

\begin{thm}\label{hbirrnotr}
Suppose that $e^{\hbar}$ is not a root of unity.

\emph{(A)}~Every representation  $T_{n,k,\ep}$ defined by~\eqref{defTnkep} is irreducible.

\emph{(B)}~Two representations  $T_{n,k,\ep}$ and $T_{n',k',\ep'}$ are equivalent only when $n=n'$, $k=k'$ and $\ep=\ep'$.

\emph{(C)}~Every continuous irreducible representation of $\wt U(\fs\fl_2)_\hbar$ on a Banach space  is finite dimensional and equivalent to some~$T_{n,k,\ep}$, where $n\in\Z_+$, $k\in\Z$ and $\ep\in\{-1,1\}$.
\end{thm}

We need a lemma.
Let $V$ be a $\wt U(\fs\fl_2)_\hbar$-module.
We say that an eigenvector~$v$ of~$H$
in~$V$  is a  \emph{weight vector} for $\wt U(\fs\fl_2)_\hbar$. The corresponding eigenvalue $\al$ is called a \emph{weight}. If,  in  addition,
$E\cdot v  = 0$,  then  $v$  is called a  \emph{highest weight vector} of weight $\al$ (cf. the versions for $U(\fs\fl_2)$ and  $U_q(\fs\fl_2)$  in \cite[p.\,101, Definition V.4.1 and  p.\,127, Definition VI.3.2]{Ka95})

\begin{lm}\label{hiwve} \emph{(cf. \cite[Lemmas V.4.3 and VI.3.4]{Ka95})}
Let $v$  be  a highest weight vector for $\wt U(\fs\fl_2)_\hbar$ of weight $\al$. For any $p\in\N$ put $v_p\!:=F^p\cdot v/[p]_q$.
Then
$$
H\cdot v_p=(\al-2p)v_p\,,\qquad E\cdot v_p=-[p-1]_{q,\la}v_{p-1}\,,\qquad F\cdot v_{p-1}=[p]_{q}v_p\,.
$$
\end{lm}
\begin{proof}
The first equality easily follows the relation $[H,F^p]=-2pF^p$. On the other hand, $v$ is a highest weight vector for $U_q(\fs\fl_2)$ of weight~$e^{\hbar\al}$. So the second and third equalities immediately  follow from \cite[Lemma VI.3.4]{Ka95}.
\end{proof}

\begin{proof}[Proof of Theorem~\ref{hbirrnotr}]
We use the representation theory of $U_q(\fs\fl_2)$.

(A)~Since $T_{n,k,\ep} \te$ coincides with the irreducible representation  $T_{n,\ep}$ of  $U_q(\fs\fl_2)$, the representation $T_{n,k,\ep}$ is also irreducible.

(B)~Note that $T_{n,k,\ep} \te$ and $T_{n',k',\ep'} \te$ are equivalent only when $n=n'$ and $\ep=\ep'$ (see \cite[p.\,128]{Ka95} or \cite[\S\,3.2, p.\,62, Proposition~8]{KSc}). On the other hand, the sets of eigenvalues of $H_{n,k,\ep}$ and $H_{n,k',\ep}$ coincide only when $k=k'$.

(C)~Let $T$ be a continuous irreducible representation of $\wt U(\fs\fl_2)_\hbar$ on a Banach space~$V$. By Theorem~\ref{fdbcoin},  $V$ is finite dimensional. It is easy to see from the relation $[H,E]=2E$ that any non-zero finite-dimensional $\wt U(\fs\fl_2)_\hbar$-module  has  a
highest weight vector~$v$ \cite[p.\,101, Proposition V.4.2]{Ka95}. Denote the corresponding weight by~$\al$. Then~$v$ is a highest weight vector for $U_q(\fs\fl_2)$ of weight~$e^{\hbar\al}$ and we can apply \cite[p.\,128, Theorem VI.3.5]{Ka95}.
In particular,  $e^{\hbar\al}=\ep e^{\hbar n}$ for some $n\in\Z_+$ and $\ep\in\{-1,1\}$, $v_p=0$ for $p>n$ and $\{v_0,v_1 , \ldots,v_n\}$  is  a  basis of an irreducible $U_q(\fs\fl_2)$-submodule~$V'$ of~$V$. Then
$\al=n+r_{k,\ep}$ for some $k\in\Z$. Finally,   Lemma~\ref{hiwve} implies that $V'$ is a $\wt U(\fs\fl_2)_\hbar$-submodule and so $V=V'$ being irreducible. Thus $T$ has the desired form.
\end{proof}

\subsection*{The structure of $\wt U(\fs\fl_2)_\hbar$ in the case when $e^{\hbar}$ is not a root of unity}

\begin{thm}\label{hbarcompred}
Suppose that $e^{\hbar}$ is not a root of unity. Then
every continuous finite-dimensional representation of $\wt U(\fs\fl_2)_\hbar$ is completely reducible.
\end{thm}
\begin{proof}
We follow the argument in \cite[Theorem VII.2.2]{Ka95} with necessary modifications.

Let $V$ be a continuous finite-dimensional $\wt U(\fs\fl_2)_\hbar$-module and $V'$ is a proper submodule of $V$. We need to show that $V'$ can be complemented

(1)~Suppose that $V'$ is of codimension~$1$. We proceed by induction on the
dimension of~$V'$.

If $\dim V'=0$ the assertion is evident.
Let $\dim V'=1$. Then $V'$ and $V/V'$  are simple one-dimensional modules of weights $\al_1$ and $\al_2$, respectively.
If $\al_1\ne \al_2$, then it is easy to see that there is a submodule complementary to $V'$.
If $\al_1= \al_2$, then there exists a basis $\{v_1, v_2 \}$  with $V' = \CC v_1$ such that
$H\cdot v_1  = \al v_1$ and $H\cdot v_2  = \al v_2 +\al' v_1$.  Since the representation is continuous,  we have that $K\cdot v_1  = e^\al v_1$ and $K\cdot v_2  = e^\al v_2 +e^\al \al' v_1$. Arguing as in \cite[Theorem VII.2.2]{Ka95}, we get that $E$ and $F$ act on $V$ trivially and, moreover,  $\al'=0$ and hence $H$ is diagonalizable. This implies again that there is a complementary submodule.

We  now  assume  that $p>1$ and  the  assertion  is
proved in dimensions  smaller than~$p$.  Let $\dim V'= p$.
If $V'$ is not simple, then it contains a submodule $V_0$ such that $\dim V_0< p$. So we can apply the induction hypothesis and deduce the assertion by a standard argument (cf. the proof of \cite[Theorem V.4.6]{Ka95}).

Suppose now that $V'$ is simple.
Recall that the quantum Casimir element
$$
C_q\!:=EF+\frac{q^{-1}K+qK^{-1}}{(q-q^{-1})^2}
$$
is central in $U_q(\fs\fl_2)$
\cite[Proposition VI.4.1]{Ka95}.
It is easy to see that $EF$ and $H$ commute and hence
$\te(C_q)$ is central in $\wt U(\fs\fl_2)_\hbar$ ($\te$ is defined in~\eqref{homthe}). Moreover, there is $\mu\in\CC$ such that $\te(C_q)+\mu$ acts by $0$ on the $1$-dimensional module $V/V'$ and by a non-zero  scalar on $V'$ \cite[Lemma VII.2.1]{Ka95}.  Arguing as in Part 1.b in the proof of \cite[Theorem VII.2.2]{Ka95}, we deduce that $V'$ can be complemented.

(2)~We now reduce the assertion of the theorem to the case of codimension~$1$. Consider the vector space $W$  of  linear maps from $V$  to $V'$ whose restriction
to $V'$ is multiplication by a scalar and the vector subspace $W'$ of  linear maps such that this scalar is~$0$. It is obvious that $W'$ has codimension $1$.

To endow $W$ and $W'$ with module  structures we need the fact that $\wt U(\fs\fl_2)_\hbar$ is a topological Hopf algebra \cite[Proposition~5.2]{AHHFG}. Denote the comultiplication, counit and antipode  by $\De$,
$\ep$ and  $S$, respectively.

Note that $\wt U(\fs\fl_2)_\hbar\ptn V$ is a $\wt U(\fs\fl_2)_\hbar\ptn \wt U(\fs\fl_2)_\hbar$-module with the multiplication determined by $(x\otimes y)\cdot (z\otimes v)\!:=xy\otimes z\cdot v$. (Here $\ptn$ denotes the complete projective tensor product of Fr\'echet spaces.)
Then the vector space of all linear maps from $V$  to $V'$ is a $\wt U(\fs\fl_2)$-module with respect to
the multiplication determined by
$$
(x\cdot f)(v)\!:=(1\otimes f)((1\otimes S)\De(x)\cdot (1\otimes v)),
$$
where  $f\!:V\to V'$  is a linear map, $x\in \wt U(\fs\fl_2)_\hbar$ and $v\in V$;
cf. \cite[\S\,III.5, p.\,58, (5.5)]{Ka95}. It is not hard to see that  $W$ and $W'$ are modules with respect to this action.

Since $W'$ has codimension $1$, it follows from Part~(1) of this proof, that there is a submodule $W''$ such that $W=W'\oplus W''$. Then there is $f$ such that $\CC f=W''$. Put $V''=\Ker f$. Then it is clear that $V=V'\oplus V''$ as a vector space.

To complete the proof it suffices to show that $V''$ is a submodule.
Since $W''$ has dimension $1$, it follows from Part~(C) of Theorem~\ref{hbirrnotr}  that
$H\cdot f =  n\pi \hbar^{-1}i f$ for some $n\in\Z$.
Since $\De(H)=H\otimes 1 + 1 \otimes H$ and $S(H)=-H$ \cite[Proposition~5.2]{AHHFG}, we have $(H\cdot f)(v)=H\cdot f(v)-f(H\cdot v)$ for every $v\in V$.
If $v\in V''$, then $f(v)=0$ and so $f(H\cdot v)=-(H\cdot f)(v)=-n\pi \hbar^{-1}i f(v)=0$. Hence $V''$ is invariant under~$H$. Arguing as at the end of the proof of \cite[Theorem VII.2.2]{Ka95}, we deduce that $V''$ is also invariant under $E$ and $F$.
\end{proof}

Denote by $\wt\Si_\hbar$ the set of equivalence
classes of continuous irreducible finite-dimensional representations of $\wt U(\fs\fl_2)_\hbar$ for given $\hbar$ (see Theorem~\ref{hbirrnotr}).
Then for $\si\in\wt\Si_\hbar$ we have a homomorphism
$\wt U(\fs\fl_2)_\hbar\to \M_{d_\si}(\CC)$, where $d_\si$ is the dimension of $\si$.
Denote by $\wt\io$ the corresponding homomorphism
$$
\wt U(\fs\fl_2)_\hbar\to \prod_{\si\in\wt\Si_\hbar}\M_{d_\si}(\CC).
$$

\begin{thm}\label{AMUhbar}\emph{(cf. Theorem~\ref{AMUq})}
Suppose that $e^{\hbar}$ is not a root of unity. Then
$\wt\io$ is a topological isomorphism.
\end{thm}
\begin{proof}
Let $\phi$ be a continuous homomorphism  from  $\wt U(\fs\fl_2)_\hbar$ to a Banach algebra and let~$B$ be the range of~$\phi$. It follows from
Theorem~\ref{hbirrnotr} that $B$ is  a finite-dimensional Banach algebra.
Then $B$ becomes a continuous $\wt U(\fs\fl_2)_\hbar$-module with respect to the action given by $a\cdot b\!:=\phi(a)b$. Since $e^{\hbar}$ is not a root of unity, it follows from Theorem~\ref{hbarcompred} that this module is completely reducible. Therefore $\phi$ factors through some finite product of algebras of the form $\M_{d_\si}(\CC)$ and hence through $\prod_{\si\in\wt\Si_\hbar}\M_{d_\si}(\CC)$. Since $\wt U(\fs\fl_2)_\hbar$  is an Arens-Michael algebra, $\wt\io$ is a topological isomorphism.
\end{proof}

\begin{rem}
The Whitehead Lemma implies that the $\hbar$-adic formal deformation of $U(\fs\fl_2)$  is isomorphic  to $U(\fs\fl_2)[[\hbar]]$ as an algebra; see~\cite[\S\,4]{Dr89}. For the analytic form we have a more subtle  relation.
Comparing the lists of irreducible Banach space representations of
$\wt{U}(\fs\fl_2)_\hbar$  and  $U(\fs\fl_2)$ (cf. Remark~\ref{Uotc2}), we have that when
$e^{\hbar}$ is not a root of unity,  there is a (non-canonical) isomorphism
$$
\wt{U}(\fs\fl_2)_\hbar\to \left(\prod_{k\in\Z} B_k \right)\otimes \CC^2
$$
of Arens-Michael algebras, where each $B_k$ is isomorphic to $\wh U(\fs\fl_2)$.
Here each multiple corresponds to a zero of the hyperbolic sine. (See a similar effect for some $2$-generated algebras in~\cite{ArRC}.) On the other hand, if $\hbar'$ is a root of unity, then  Theorem~\ref{rouirrfdhb} implies that $\wt{U}(\fs\fl_2)_{\hbar'}$ is not isomorphic to $\wt{U}(\fs\fl_2)_\hbar$.
\end{rem}

\section{Concluding remarks and questions}
\label{sec:remqu}

\subsection*{Representations in the exceptional case}
We show in Theorem~\ref{fdbana} that all the irreducible Banach space representations of $U_q(\fg)$, where
$\fg$ be a  semisimple complex  Lie algebra and $|q|\ne 1$, are finite-dimensional. Moreover, it is well known that this is true (not only in the Banach space case) when $q$ is a root of unity. As a corollary, the irreducible Banach space representations can be classified in both cases. On the other hand, in the third case when $|q|=1$ and $q$ is not a root of unity, there are infinite-dimensional topologically irreducible representations; see Proposition~\ref{Tla1topirr}.
But the following two questions are open.

\begin{que}
Suppose that $|q|=1$ and $q$ is not a root of unity.
Are there any infinite-dimensional irreducible Banach space representations of $U_q(\fg)$? In particular, is it true that the representation $S_{\la,2}$ of $U_q(\fs\fl_2)$ considered in Proposition~\ref{Tla1topirr} is irreducible when~$q^2$ is not a root of~$\la^2$?
\end{que}

\begin{que}
Suppose that $|q|=1$ and $q$ is not a root of unity.
Is it possible to give a reasonable classification of topologically irreducible representations of $U_q(\fg)$ on Banach spaces or at least a classification of topologically simple Banach algebras that are completions of~$U_q(\fg)$?
\end{que}

\subsection*{Injectivity}
Since $U_q(\fg)$ has infinite-dimensional irreducible Banach space representations, some information is lost
with applying the Arens-Michael enveloping homomorphism $\io\!:U_q(\fg)\to \wh U_q(\fg)$. The same can be said for
the natural homomorphism $\te\!:U_q(\fg)\to \wt U(\fg)_\hbar$, where $q=e^\hbar$.  (For $\wt U(\fg)_\hbar$ see \cite[Remark~5.5]{AHHFG}.  When $\fg=\fs\fl_2$ the definition of $\te$ is given for in~\eqref{homthe}; in the general case
it is defined in a similar way.) But there is hope at least that the kernels of $\io$ and $\te$ are trivial.

\begin{que}\label{injh}
Are the homomorphisms  $\io\!:U_q(\fg)\to \wh U_q(\fg)$ and  $\te\!:U_q(\fg)\to \wt U(\fg)_\hbar$  always injective?
\end{que}

The first part is a partial case of Question~2 in \cite{AHHFG}.

\begin{rem}
In the classical case, the assertion that the Arens-Michael enveloping homomorphism $U(\fg)\to \wh U(\fg)$ is injective can  easily be deduced from the well-known fact that the adjoint representation of $U(\fg)$ is faithful and locally finite. When $q$ is not a root of unity, a similar argument cannot be applied to $U_q(\fg)$  because  the adjoint representation is not locally finite (see \cite[p.\,153]{Sm92}). On the other hand, by the result of Pedchenko \cite{Pe15} mentioned in the introduction, the map
$\cO(\CC\times\CC^\times\times \CC)\to  \wh U_q(\fs\fl_2)$ is a topological isomorphism when $|q|=1$. Therefore $\io\!:U_q(\fs\fl_2)\to \wh U_q(\fs\fl_2)$  is injective under this assumption.
\end{rem}

\subsection*{Complete reducibility}
The main step in the proof of the structural result for $\wt U(\fs\fl_2)_\hbar$ in the case when $e^{\hbar}$ is not a root of unity is the complete reducibility of every continuous finite-dimensional representation (Theorem~\ref{hbarcompred}).
To provide a structural result for $\wt U(\fg)_\hbar$ for arbitrary~$\fg$ we need a similar assertion on complete reducibility.

\begin{que}
Suppose that $e^{\hbar}$ is not a root of unity. Is every continuous finite-dimensional representation of $\wt U(\fg)_\hbar$ completely reducible?
\end{que}


\begin{thebibliography}{CCJJVVXX}

\bibitem[Ar20]{ArAMN}
O.\,Yu.~Aristov, \emph{Arens-Michael envelopes of nilpotent Lie
algebras, functions of exponential type, and homological
epimorphisms}, Tr. Mosk. Mat. Obs Trans. 81:1 (2020), 117--136,  Trans. Moscow Math. Soc. 2020, 97--114, arXiv:1810.13213.

\bibitem[Ar20+]{AHHFG}
O.\,Yu.~Aristov, \emph{Holomorphically finitely generated Hopf algebras and quantum Lie groups}, arXiv:2006.12175 (2020).

\bibitem[Ar21]{ArRC}
O.\,Yu.~Aristov, \emph{The relation ``commutator equals function'' in Banach algebras}, Mat. Zametki, 109:3 (2021), 323--337 (Russian), English transl.: Math. Notes, 109:3 (2021), 323--334,  arXiv:1911.03293.

\bibitem[BS01]{BS01}
D.~Belti\c{t}\u{a}, M.~\c{S}abac, \emph{Lie algebras of bounded operators}, Birkh\"{a}user,  Basel, Boston, Berlin,   2001.

\bibitem[BG02]{BG02}
K.\,A.~Brown,  K.\,R.~Goodearl, \emph{Lectures on Algebraic Quantum Groups}. Series: Advanced courses in mathematics, CRM Barcelona, Birkhäuser: Basel, 2002.

\bibitem[Dr89]{Dr89}
 V.\,G.~Drinfeld, \emph{Almost cocommutative Hopf algebras}, Algebra i Analiz, 1:2 (1989), 30--46; Leningrad Math. J., 1:2 (1990), 321--342.

\bibitem[Du19]{Du19}
N.~Dupre, \emph{Rigid analytic quantum groups and quantum
Arens-Michael envelopes}, J. Algebra 537 (2019), 98--146.

\bibitem[He93]{X2}
A.\,Ya.~Helemskii, \emph{Banach and polynormed algebras: General theory,
representations, homology}, Nauka, Moscow 1989 (Russian); English transl.:
Oxford University Press, 1993.

\bibitem[Ka95]{Ka95}
C.~Kassel, \emph{Quantum Groups}, Springer, 1995.

\bibitem[KS97]{KSc}
A.\,U.~Klimyk, K.~Schm\"{u}dgen, \emph{Quantum groups and their representations}, Springer, Berlin 1997.

\bibitem[Lu93]{Lu93}
G.~Lusztig, \emph{Introduction to quantum groups}, Birkh\"{a}user, Boston, 1993 .

\bibitem[Ly13+]{Lu13}
A.~Lyubinin, \emph{$p$-adic quantum hyperenveloping algebra for $\fs\fl_2$},
 arXiv:1312.4372, 2013.

\bibitem[Pi06]{Pi4}
A.\,Yu.~Pirkovskii, \emph{Arens-Michael enveloping algebras and analytic smash
products}, Proc. Amer. Math. Soc. 134 (2006), 2621--2631.

\bibitem[Pi08]{Pir_qfree}
A.\,Yu.~Pirkovskii, \emph{Arens-Michael envelopes, homological epimorphisms, and
relatively quasi-free algebras}, (Russian), Tr. Mosk. Mat. Obs. 69
(2008), 34--125; English translation in Trans. Moscow Math. Soc. (2008), 27--104.

\bibitem[Pi11]{Pi11}
A.\,Yu.~Pirkovskii,
\emph{The Arens-Michael envelope of a smash product}, arXiv:1101.0166.

\bibitem[Pi14]{Pi14}
A.\,Yu.~Pirkovskii, \emph{Noncommutative analogues of Stein spaces of
finite embedding dimension}, Algebraic methods in functional
analysis: the Victor Shulman anniversary volume, Operator theory
Aadvances and applications, 233, ed. Todorov, Turowska, Birkhauser 2014, 135--153.

\bibitem[Pi15]{Pi15}
A.\,Yu.~Pirkovskii, \emph{Holomorphically finitely generated
algebras}, J. Noncommutative Geom. 9:1 (2015), 215--264.

\bibitem[Pe15]{Pe15}
D.~Pedchenko, \emph{Arens-Michael envelopes of Jordanian plane and $U_q(\fs\fl(2))$}, Master Thesis, National Research University
Higher School of Economics,
Department of Mathematics, Moscow, 2015.

\bibitem[Ro88]{Ro88}
M.~Rosso, \emph{Finite dimensional representations of the quantum analog of the
enveloping algebra of a complex simple Lie algebra}, Commun. Math. Phys. 117
(1988), 581--593.

\bibitem[Sm92]{Sm92}
S.\,P.~Smith, \emph{Quantum groups: an introduction and survey for ring theorists} in: S.~ Montgomery,  L.~ Small eds. \emph{Noncommutative Rings}, MSRI Publ. 24, Springer, Berlin 1992, 131--178.

\bibitem[Sm18+]{Sm18}
C.~Smith, \emph{On analytic analogues of quantum groups}, arXiv:1806.10502, 2018.


\bibitem[Ta72]{T2}
J.\,L.~Taylor, \emph{A general framework for a multi-operator functional calculus},
Adv. Math. 9 (1972), 183--252.

\bibitem[We52]{We52}
J.~Wermer, \emph{On invariant subspaces of normal operators}, Proc. Amer. Math. Soc. 3 (1952), 270--277.
\end{thebibliography}
\end{document}